\documentclass[11pt]{amsart}

\usepackage{amsfonts,epsfig}
\usepackage{latexsym}
\usepackage{amssymb}
\usepackage{amsmath}
\usepackage{amsthm}
\usepackage{graphics}
\usepackage[all]{xy}
\usepackage[T2A]{fontenc}
\usepackage{multirow}
\usepackage{hhline}
\usepackage{array, booktabs, ctable}
\usepackage[backref]{hyperref}
\usepackage{enumerate}
\usepackage{comment}
\usepackage{breqn}

\newtheorem{defn}{Definition}[section]

\newtheorem{corollary}[defn]{Corollary}
\newtheorem{lemma}[defn]{Lemma}

\newtheorem{thm}[defn]{Theorem}
\newtheorem{theorem}[defn]{Theorem}
\newtheorem{cor}[defn]{Corollary}
\newtheorem{prop}[defn]{Proposition}

\newtheorem*{theorem*}{Theorem}

\theoremstyle{definition}

\newtheorem*{ack}{Acknowledgements}
\newtheorem{remark}[defn]{Remark}

\newcommand{\CC}{\mathbb C}

\newcommand{\Q}{\mathbb Q}
\newcommand{\Z}{\mathbb Z}

\newcommand{\OK}{\mathcal{O}_K}
\newcommand{\Of}{\mathcal{O}_{K,f}}
\newcommand{\cC}{\mathcal{C}}
\newcommand{\OO}{\mathcal{O}}

\newcommand{\Rats}{\mathbb{Q}}

\newcommand{\Gal}{\operatorname{Gal}}
\newcommand{\Aut}{\operatorname{Aut}}
\newcommand{\GQ}{\Gal(\overline{\Rats}/\Rats)}

\newcommand{\GL}{\operatorname{GL}}

\begin{document}



\title[Abelian Division Fields and CM]{Elliptic curves with  complex multiplication and abelian division fields}

\author{Asimina S. Hamakiotes}
\address{University of Connecticut, Department of Mathematics, Storrs, CT 06269, USA}
\email{asimina.hamakiotes@uconn.edu} 
\urladdr{https://asiminah.github.io/}

\author{\'Alvaro Lozano-Robledo}
\address{University of Connecticut, Department of Mathematics, Storrs, CT 06269, USA}
\email{alvaro.lozano-robledo@uconn.edu} 
\urladdr{http://alozano.clas.uconn.edu}





\begin{abstract} Let $K$ be an imaginary quadratic field, and let $\mathcal{O}_{K,f}$ be an order in $K$ of conductor $f\geq 1$. Let $E$ be an elliptic curve with CM by $\mathcal{O}_{K,f}$, such that $E$ is defined by a model over $\mathbb{Q}(j_{K,f})$, where $j_{K,f}=j(E)$. In this article, we classify the values of $N\geq 2$ and the elliptic curves $E$ such that (i) the division field $\Q(j_{K,f},E[N])$ is an abelian extension of $\Q(j_{K,f})$, and (ii) the $N$-division field coincides with the $N$-th cyclotomic extension of the base field.
\end{abstract}

\maketitle



\section{Introduction}
Let $F$ be a number field and let $E$ be an elliptic curve defined over $F$. Let $\overline{F}$ be a fixed algebraic closure of $F$, let $N\geq 2$ be an integer, and let $E[N]=E(\overline{F})[N]$ be the $N$-torsion subgroup of $E(\overline{F})$, i.e., the subgroup formed by all the $\overline{F}$-rational points on $E$ of finite order dividing $N$. It is then natural to study the extension $F(E[N])/F$ where $F(E[N])$ denotes the field of definition of the coordinates of points in $E[N]$, which is usually called the $N$-division field of $E/F$. Since the extension $F(E[N])/F$ is Galois, it is also natural to study its Galois group $\Gal(F(E[N])/F)$, and the Galois representations that arise from the natural action of Galois on the $N$-torsion subgroup of the elliptic curve $E$.

The existence of the Weil pairing implies that $F(E[N])$ contains all the $N$-th roots of unity of $\overline{F}$, i.e., $F(\zeta_N)\subseteq F(E[N])$, where $\zeta_N$ is a primitive $N$-th root of unity (\cite[Ch. III, Cor. 8.1.1]{silverman1}). Here we are interested in classifying all number fields $F$ and elliptic curves $E/F$ with one of the following properties:
\begin{enumerate}
    \item  The extension $F(E[N])/F$ is abelian, i.e., $\Gal(F(E[N])/F)$ is abelian.
    \item The $N$-division field equals the $N$-th cyclotomic field, i.e., $F(E[N])=F(\zeta_N)$.
\end{enumerate}

 Previously, the work of Halberstadt, Merel \cite{merel}, Merel and Stein \cite{merel-stein}, and Rebolledo \cite{rebolledo} showed that if $p$ is prime, and $F(E[p])=\Q(\zeta_p)$, then $p=2,3,5$ or $p>1000$. When $F=\Q$,  Paladino \cite{paladino} classified all the curves $E/\Q$ with $\Q(E[3])=\Q(\zeta_3)$.
In \cite{gonzalez-jimenez-lozano-robledo}, Gonz\'alez-Jim\'enez and the second author proved that $\Q(E[N])=\Q(\zeta_N)$ only for $N=2,3,4,$ or $5$. More generally, they proved that $\Q(E[N])/\Q$ can be abelian only for $N=2,3,4,5,6,$ or $8$, and classified all the elliptic curves $E/\Q$ with this property. Moreover, for each $N$, they classified the possible abelian Galois groups that occur for a division field. The classification of abelian division fields of elliptic curves over $\Q$ has numerous applications, for example in: the classification of torsion subgroups of elliptic curves \cite{chou1}, \cite{enrique}, \cite{chou2}, \cite{chou3}; the classification of isogeny-torsion graphs \cite{chiloyan1}, \cite{chiloyan2}; Brauer groups \cite{varillyviray}; non-monogenic number fields \cite{hanson1}, \cite{hanson2}; congruences between elliptic curves \cite{cremona}; or the classification of Galois representations \cite{lozano-galoiscm}, \cite{elladic}, among others.

The goal of this article is a complete classification of all elliptic curves with complex multiplication defined over their minimal field of definition that satisfy (1) or (2).  Let $K$ be an imaginary quadratic field, and let $\OK$ be the ring of integers of $K$ with discriminant $\Delta_K$. Let $f\geq 1$ be an integer and let $\mathcal{O}_{K,f}$ be the order of $K$ of conductor $f$. Let $j(\Lambda)$ be the modular $j$-invariant function on lattices $\Lambda\subseteq \CC$, and let $j(\mathbb{C}/\mathcal{O}_{K,f})$ be the $j$-invariant associated to the order $\mathcal{O}_{K,f}$ when regarded as a complex lattice. The theory of complex multiplication shows that $j(\mathbb{C}/\mathcal{O}_{K,f})$ is an algebraic integer (see \cite[Ch. 2, \S 6]{silverman2}). Let $j_{K,f}$ be an arbitrary Galois conjugate of $j(\mathbb{C}/\mathcal{O}_{K,f})$. Then, an elliptic curve $E/\mathbb{Q}(j_{K,f})$ with $j(E)=j_{K,f}$ has complex multiplication by $\mathcal{O}_{K,f}$, and we say that $\Q(j_{K,f})$ is a minimal field of definition for elliptic curves with CM by $\mathcal{O}_{K,f}$ (note that there are several such conjugate fields). The main results of this paper are as follows.

The first main theorem classifies those curves with CM that satisfy (1), i.e., those curves with CM such that the $N$-th division field is abelian over the field of definition.

\begin{theorem}\label{main_thm}\label{thm-main}\label{thm-main1}
    Let $E/\mathbb{Q}(j_{K,f})$ be an elliptic curve with CM by $\mathcal{O}_{K,f}$, $f\geq 1$. Let $N\geq 2$ and let $G_{E,N}=\Gal(\Q(j_{K,f},E[N])/\Q(j_{K,f}))$ be the Galois group of the $N$-division field of $E$. Then $G_{E,N}$ is only abelian for $N=2,3$, and $4$, and $G_{E,N}\cong (\Z/2\Z)^k$ for some $0\leq k\leq 3$. Moreover:

\begin{enumerate}
        \item If $N=2$, then $G_{E,2}$ is abelian if and only if one of the following holds: 
        \begin{enumerate} \item $j_{K,f}\neq 0,1728$ and either
      \begin{itemize}
            \item $\Delta_Kf^2\equiv 0 \bmod 4$, or 
            \item  $\Delta_K\equiv 1 \bmod 8$ and $f\equiv 1 \bmod 2$.
        \end{itemize}
        In this case $G_{E,2}\cong \Z/2\Z$.
        \item $j_{K,f}=1728$. In this case $G_{E,2}$ is always abelian with $G_{E,2}\cong \{0\}$ or $\Z/2\Z$ according to whether $E$ is given by $y^2=x^3-dx$ with $d$ a square or a non-square in $\Z$, respectively.
        \item $j_{K,f}=0$ and $E/\Q$ is given by $y^2=x^3+d$ with $d$ a cube in $\Z$. In this case $G_{E,2}\cong \Z/2\Z$.
        \end{enumerate}

        \item If $N=3$, then $G_{E,3}$ is abelian if and only if $j(E)=0$ and $E/\Q$ is given by $y^2=x^3+d$ such that $-4d$ is a cube in $\Z$. If in addition $d$ and $-3d$ are not squares, then  $G_{E,3}\cong (\Z/2\Z)^2$, and  if $d$ or $-3d$ is a square, then $G_{E,3}\cong \Z/2\Z$.

        \item If $N=4$, then $G_{E,4}$ is abelian if and only if $j(E)=1728$ and $E/\Q$ is given by $y^2=x^3+dx$ with $d\in\{\pm 1,\pm 4\}$ or $d=\pm t^2$ for some square-free integer $t\not\in \{\pm 1,\pm 2\}$, in which case $G_{E,4}\cong (\Z/2\Z)^2$ or $(\Z/2\Z)^3$, respectively.

\end{enumerate}

\end{theorem}

Our second main theorem classifies those curves with CM that satisfy property (2), i.e., curves with CM such that the $N$-th division field is cyclotomic over their minimal field of definition.

\begin{theorem}\label{thm-main2}
    Let $E/\mathbb{Q}(j_{K,f})$ be an elliptic curve with CM by $\mathcal{O}_{K,f}$, with $f \geq 1$, and let $N\geq 2$. Then, $\mathbb{Q}(j_{K,f}, E[N]) = \mathbb{Q}(j_{K,f},\zeta_N)$ if and only if 
    \begin{enumerate}
        \item $N=2$, and $j(E)=1728$ with $E/\Q$ given by a form $y^2=x^3-d^2x$ for some integer $d$, or
        \item $N=3$, and $j(E)=0$ with $E/\Q$ given by a form $y^2=x^3+d$ for some integer $d$ such that $d$ or $-3d$ is a square and $-4d$ is a cube. 
    \end{enumerate}
\end{theorem} 

\subsection{Structure of the paper} The paper is organized as follows. In Section \ref{examples} we provide some examples of curves that satisfy the conditions of Theorems \ref{thm-main1} and \ref{thm-main2}. In Section \ref{proofs} we recall some basic results about the image of Galois representations attached to CM elliptic curves, summarize the main results from \cite{lozano-galoiscm}, and prove some preliminary results to understand what subgroups of normalizers of Cartan subgroups are  abelian. In Section \ref{applying_results} we apply the results from Section \ref{proofs} to the subgroups of $\GL(2,\Z_p)$ that can occur as images of $\rho_{E,p^\infty}$ from \cite{lozano-galoiscm} and analyze under what conditions we have that $G_{E,N}=\operatorname{im} \rho_{E,N}$ is abelian in the prime power case. In Section \ref{proof-thm-main1} we complete the proof of Theorem \ref{thm-main1} and in Section \ref{proof-thm-main2} we complete the proof of Theorem \ref{thm-main2}. 

\begin{ack}
    The authors would like to thank Enrique Gonz\'alez-Jim\'enez for helpful comments and suggestions.
\end{ack}

\section{Examples}\label{examples}

In this section we provide a few examples of curves (with LMFDB labels, \cite{lmfdb}) that satisfy the conditions of Theorems \ref{thm-main1} and \ref{thm-main2}.
\subsection{$N=2$}
\begin{itemize}
    \item $E_1/\Q: y^2=x^3-x$ (\href{http://www.lmfdb.org/EllipticCurve/Q/32/a/3}{32.a3}) has abelian $2$-division field with $G_{E_1,2}\cong \{0\}$, and $j(E_1)=1728$.
    \item $E_2/\Q: y^2=x^3-2x$ (\href{http://www.lmfdb.org/EllipticCurve/Q/256/b/1}{256.b1}) has $G_{E_2,2}\cong \Z/2\Z$, and $j(E_2)=1728$.
    \item $E_3/\Q: y^2=x^3+1$ (\href{http://www.lmfdb.org/EllipticCurve/Q/36/a/4}{36.a4}) has $G_{E_3,2}\cong \Z/2\Z$, and $j(E_3)=0$.
    \item $E_4/\Q: y^2=x^3+x^2-13x-21$ (\href{http://www.lmfdb.org/EllipticCurve/Q/256/a/1}{256.a1}) has $G_{E_4,2}\cong \Z/2\Z$, and $j(E_4)=8000$. Here $\Delta_K=-8$ and $f=1$, so $\Delta_K f^2 \equiv 0 \bmod 4$.
    \item $E_5/\Q: y^2+xy=x^3-x^2-107x+552$ (\href{http://www.lmfdb.org/EllipticCurve/Q/49/a/2}{49.a2}) has $G_{E_5,2}\cong \Z/2\Z$, and $j(E_5)=-3375$. Here $\Delta_K=-7\equiv 1 \bmod 8$ and $f=1$.

\item In addition, here is an example of an elliptic curve that is not defined over $\Q$. Let $E_6$ be given by
$${y}^2+\sqrt{2}{x}{y}={x}^{3}+{x}^{2}+\left(15\sqrt{2}-22\right){x}+46\sqrt{2}-69$$
which is defined over $\Q(\sqrt{2})$. This is the elliptic curve \href{http://www.lmfdb.org/EllipticCurve/2.2.8.1/32.1/a/1}{32.1-a1} defined over $\Q(\sqrt{2})$. We note that $j(E_6)= -29071392966 \sqrt{2} + 41113158120$ so there is no model over $\Q$. Indeed, this curve has complex multiplication by $\Z[\sqrt{-16}]$ which is the suborder of conductor $f=4$ of $\OK=\Z[i]$. In particular, $\Delta_K f^2 = -4\cdot 16 = -64\equiv 0 \bmod 4$, so Theorem 1.1 implies that $G_{E_6,2}\cong \Z/2\Z$. This fact can be verified by checking that $E_6(\Q(\sqrt{2}))[2]\cong \Z/2\Z$ is generated by a non-trivial point of two torsion defined over $\Q(\sqrt{2})$, namely
$$P=\left(2 \sqrt{2} - \frac{3}{2} , \frac{3}{4} \sqrt{2} - 2 \right).$$
\item Similarly, let $E_7$ be given by
$${y}^2+{x}{y}+\varphi{y}={x}^{3}-{x}^{2}-2\varphi{x}+\varphi$$
where $\varphi=(1+\sqrt{5})/2$ is the golden ratio, so that $E_7$ is defined over $\Q(\sqrt{5})$. This is the curve \href{http://www.lmfdb.org/EllipticCurve/2.2.5.1/81.1/a/1}{81.1-a1} defined over $\Q(\sqrt{5})$. We note that $j(E_7)= -85995 \varphi - 52515$ so $E_7$ has no model over $\Q$. This curve has CM by $\Z[(1+\sqrt{-15})/2]$ and therefore $\Delta_K=-15$ and $f=1$. In particular, $\Delta_K\equiv 1 \bmod 8$ and Theorem 1.1 implies that $G_{E_7,2}\cong \Z/2\Z$. Once again this can be verified: $E(\Q(\sqrt{5}))[6]\cong \Z/6\Z$ so there is a unique non-trivial $2$-torsion point defined over $\Q(\sqrt{5})$, namely $Q = \left(-\varphi + 3 , 3 \varphi - 6 \right).$
\end{itemize}

\subsection{$N=3$}
\begin{itemize} 
    \item $E_8/\Q: y^2=x^3+2$ (\href{https://www.lmfdb.org/EllipticCurve/Q/1728/n/4}{1728.n4}) has $G_{E_8,3}\cong (\Z/2\Z)^2$, and $j(E_8)=0$. Here $-4d=-4\cdot 2=-8$ is a cube, but $d=2$ and $-3d=-6$ are not cubes.
    \item $E_9/\Q: y^2 = x^3 + 16$ (\href{https://www.lmfdb.org/EllipticCurve/Q/27/a/4}{27.a4}) has $G_{E_9,3}\cong \Z/2\Z$, and $j(E_9)=0$. Here $-4d=-4\cdot 16=-64=-2^6$ is a cube and $d=16$ is a square. Moreover, here we have $\Q(E[3])=\Q(\sqrt{-3})=\Q(\zeta_3)$.
\end{itemize}

\subsection{$N=4$}
\begin{itemize}
    \item $E_{10}/\Q: y^2=x^3-4x$ (\href{https://www.lmfdb.org/EllipticCurve/Q/64/a/3}{64.a3}) has $G_{E_{10},4}\cong (\Z/2\Z)^2$, and $j(E_{10})=1728$. Here $d=-4$.
    \item $E_{11}/\Q: y^2=x^3+9x$ (\href{https://www.lmfdb.org/EllipticCurve/Q/576/c/4}{576.c4}) has $G_{E_{11},4}\cong (\Z/2\Z)^3$, and $j(E_{11})=1728$. Here $d=3^2$. 
\end{itemize}


\section{Background and Preliminary Results}\label{proofs}

In this section, we will prove some preliminary results. For simplicity, we will establish some notation. Let $K$ be an imaginary quadratic field, and let $\OK$ be the ring of integers of $K$ with discriminant $\Delta_K$. Let $f\geq 1$ be an integer and let $\mathcal{O}_{K,f}$ be the order of $K$ of conductor $f$. Let $j_{K,f}$ be a CM $j$-invariant as in the introduction and let $N\geq 2$. We define associated constants $\delta$ and $\phi$ as follows:
	\begin{itemize}
\item If $\Delta_Kf^2\equiv 0\bmod 4$, let $\delta=\Delta_K f^2/4$, and $\phi=0$.
		\item If $\Delta_Kf^2\equiv 1 \bmod 4$, let $\delta=\frac{(\Delta_K-1)}{4}f^2$, let $\phi=f$.
	\end{itemize}
We define matrices in $\GL(2,\Z/N\Z)$ by
$$c_\varepsilon=\left(\begin{array}{cc} -\varepsilon & 0\\ \phi & \varepsilon\\\end{array}\right), \quad \text{ and } \quad c_{\delta,\phi}(a,b)=\left(\begin{array}{cc}a+b\phi & b\\ \delta b & a\\ \end{array}\right)$$ where $\varepsilon \in \{\pm 1\}$ and $a,b \in \Z/N\Z$ such that $\det(c_{\delta,\phi}(a,b))\in (\Z/N\Z)^\times$. 
We define the Cartan subgroup $\cC_{\delta,\phi}(N)$ of $\GL(2,\Z/N\Z)$ by
$$\cC_{\delta,\phi}(N)=\left\{c_{\delta,\phi}(a,b): a,b\in\Z/N\Z,\  \det(c_{\delta,\phi}(a,b)) \in (\Z/N\Z)^\times \right\},$$
and $\mathcal{N}_{\delta,\phi}(N) = \left\langle \cC_{\delta,\phi}(N),c_1\right\rangle$. We will sometimes call $\mathcal{N}_{\delta,\phi}(N)$ the ``normalizer'' of $\cC_{\delta,\phi}(N)$ in $\GL(2,\Z/N\Z)$ but we emphasize here that, while this is the case when $N$ is prime, it may not be the actual normalizer for composite values of $N$ (see \cite{lozano-galoiscm}, Section 5). Finally, we write $\mathcal{N}_{\delta,\phi} = \varprojlim \mathcal{N}_{\delta,\phi}(N)$.

First, we recall results from \cite{lozano-galoiscm} that provide details about the image of Galois representations attached to CM elliptic curves.

\begin{theorem}[\cite{lozano-galoiscm}, Theorems 1.1 and 1.2]\label{thm-cmrep-intro-alvaro}
	Let $E/\Q(j_{K,f})$ be an elliptic curve with CM by $\OO_{K,f}$, let $N\geq 3$, and let $\rho_{E,N}$ be the Galois representation $\Gal(\overline{\Q(j_{K,f})}/\Q(j_{K,f})) \to \Aut(E[N])\cong \GL(2,\Z/N\Z)$, and let $\rho_{E}$ be the Galois representation $\Gal(\overline{\Q(j_{K,f})}/\Q(j_{K,f})) \to \varprojlim \Aut(E[N])\cong \GL(2,\widehat{\Z})$. Then,
	\begin{enumerate}
            \item There is a $\Z/N\Z$-basis of $E[N]$ such that the image of $\rho_{E,N}$
	is contained in $\mathcal{N}_{\delta,\phi}(N)$, and	 the index of the image of $\rho_{E,N}$ in $\mathcal{N}_{\delta,\phi}(N)$ is a divisor of the order of $\Of^\times/\mathcal{O}_{K,f,N}^\times$, where $\mathcal{O}_{K,f,N}^\times=\{u\in\Of^\times: u\equiv 1 \bmod N\Of\}$.
		\item There is a compatible system of bases of $E[N]$ such that the image of $\rho_E$ is contained in $\mathcal{N}_{\delta,\phi}$, and the index of the image of $\rho_{E}$ in $\mathcal{N}_{\delta,\phi}$ is a divisor of the order $\Of^\times$. In particular, the index is a divisor of $4$ or $6$. 
	\end{enumerate} 
\end{theorem}

Our goal is to understand what subgroups of $\mathcal{N}_{\delta,\phi}(N)$ are images of Galois representations and abelian. First, we recall that the Cartan subgroup $\cC_{\delta,\phi}(N)$ is abelian.

\begin{lemma}\label{lem-cartanisabelian}
    Let $N\geq 2$. The group $\cC_{\delta,\phi}(N)$ is isomorphic to $(\mathcal{O}_{K,f}/N\mathcal{O}_{K,f})^\times$ so, in particular, it is an abelian subgroup of $\GL(2,\Z/N\Z)$.
\end{lemma}
\begin{proof}
    See \cite{lozano-galoiscm}, Lemma 2.5 and Remark 2.6.
\end{proof}

Next, we give conditions that will help us characterize when a subgroup of  $\mathcal{N}_{\delta,\phi}(N)$ is abelian. 

\begin{lemma}\label{abelian_conditions}
     Let $N\geq 2$ and let  $G\subseteq \mathcal{N}_{\delta,\phi}(N)$ be a subgroup. If $c_1, c_{\delta,\phi}(a,b) \in G$, for some $a,b\in \Z/N\Z$, such that $c_1 \cdot c_{\delta,\phi}(a,b) = c_{\delta,\phi}(a,b) \cdot c_1$, then $b\phi \equiv 0 \bmod N$ and $2b\equiv 0 \bmod N$.  Moreover, if $\phi=0$, and if $c_\varepsilon,c_{\delta,0}(a,b)\in G$ for some $\varepsilon \in \{\pm 1\}$, such that $c_\varepsilon \cdot c_{\delta,0}(a,b) = c_{\delta,0}(a,b) \cdot c_\varepsilon$ then the same conclusion holds.
     \end{lemma}

\begin{proof}
    Suppose that $c_1$ and $c_{\delta,\phi}(a,b)\in G$, where $a,b\in \mathbb{Z}/N\mathbb{Z}$ and $a^2+ab\phi-\delta b^2 \in (\mathbb{Z}/N\mathbb{Z})^\times$, and suppose the two matrices commute:
    \begin{align*}
        \left(\begin{array}{cc} a+b\phi & b \\ \delta b & a\\\end{array}\right)\left(\begin{array}{cc} -1 & 0\\ \phi & 1\\\end{array}\right) \ &\equiv \ \left(\begin{array}{cc} -1 & 0\\ \phi & 1\\\end{array}\right)\left(\begin{array}{cc} a+b\phi & b \\ \delta b & a\\\end{array}\right) \ \bmod N, \text{ and so} \\
        \left(\begin{array}{cc} -a & b\\ -\delta b + a\phi & a\\\end{array}\right) \ &\equiv \ \left(\begin{array}{cc} -a-b\phi & -b \\ a\phi + b\phi^2 + \delta b & b\phi + a\\\end{array}\right) \ \bmod N.
    \end{align*}
    In order for the last equivalence to hold, we must have that
    \begin{itemize}
        \item $-a \equiv -a-b\phi \bmod N$ and therefore $ b\phi \equiv 0 \bmod N$, 
        \item $b \equiv -b \bmod N$ and so $2b \equiv 0 \bmod N$,
        \item $-\delta b + a\phi \equiv a \phi + b\phi^2 + \delta b \bmod N$ which implies $b\phi^2 + 2\delta b \equiv 0 \bmod N$, and
        \item $a \equiv b\phi + a \bmod N$ which also implies that $b\phi \equiv 0 \bmod N$.
    \end{itemize}
    In particular, $b\phi \equiv 0 \bmod N$ and $2b \equiv 0 \bmod N$, as desired. If $\phi=0$, we note that $c_\varepsilon = (\varepsilon\cdot \operatorname{Id})\cdot c_1$, so $c_\varepsilon \cdot c_{\delta,0}(a,b) = c_{\delta,0}(a,b) \cdot c_\varepsilon$ implies $b\phi \equiv 0 \bmod N$ and $2b \equiv 0 \bmod N$ as well.
\end{proof}

As a corollary, we obtain a simple way to prove that a subgroup is non-abelian. 

\begin{corollary}\label{nonabelian}
Let $N> 2$ and let  $G\subseteq \mathcal{N}_{\delta,\phi}(N)$ be a subgroup. If $c_1\in G$ (or $\phi=0$ and $c_\varepsilon\in G$) and $c_{\delta,\phi}(a,b)\in G$ with $b\in (\Z/N\Z)^\times$, then $G$ is non-abelian.
\end{corollary}

\begin{proof}
    Assume for a contradiction that $G\subseteq \mathcal{N}_{\delta,\phi}(N)$ is abelian. Then, $c_1$ (or $c_\varepsilon$ if $\phi=0$) and $c_{\delta,\phi}(a,b)$ commute, and by Lemma \ref{abelian_conditions} we have that $b\phi \equiv 0 \bmod N$ and $2b \equiv 0 \bmod N$. If $N>2$ and $b\in (\Z/N\Z)^\times$, then the equivalence $2b \equiv 0 \bmod N$ implies that $2 \equiv 0 \bmod N$, which is not possible since $N>2$. Therefore, $G$ cannot be abelian.  
\end{proof}

The following lemma determines the isomorphism type of $\mathcal{N}_{\delta,\phi}(2) $.

\begin{lemma}\label{lem-abelianincase2} Let $N=2$.
$$\mathcal{N}_{\delta,\phi}(2) \cong \begin{cases}
    S_3 & \text{ if } \delta\equiv \phi \equiv 1 \bmod 2, \text{ or}\\
    \Z/2\Z & \text{ otherwise.}
\end{cases}$$
\end{lemma}
\begin{proof}
The following inclusion of matrices with coefficients in $\Z/2\Z$
$$\mathcal{N}_{\delta,\phi}(2)\subseteq \left\{\left(\begin{array}{cc} 1 & 0 \\ 0 & 1\\\end{array}\right), \left(\begin{array}{cc} 1+\phi & 1 \\ \delta  & 1\\\end{array}\right),\left(\begin{array}{cc} \phi & 1 \\ \delta  & 0\\\end{array}\right),\left(\begin{array}{cc} 1 & 0\\ \phi & 1\\\end{array}\right),\left(\begin{array}{cc} 1 & 1\\ \delta+\phi & 1\\\end{array}\right),\left(\begin{array}{cc} 0 & 1\\ \delta & 0\\\end{array}\right)\right\}$$
implies that the only elements of order $3$ in $\GL(2,\Z/2\Z)$, which are $\left(\begin{array}{cc} 1 & 1 \\ 1 & 0\\\end{array}\right)$ and $\left(\begin{array}{cc} 0 & 1 \\ 1 & 1\\\end{array}\right)$, only belong to $\mathcal{N}_{\delta,\phi}(2)$ when $\delta\equiv \phi\equiv 1\bmod 2$ (and if this is the case, all the matrices are distinct mod $2$ and $\mathcal{N}_{\delta,\phi}(2)\cong S_3$). Thus, if $\delta$ or $\phi\equiv 0 \bmod 2$, then $\mathcal{N}_{\delta,\phi}(2)$ is either trivial or isomorphic to $\Z/2\Z$ but at least one of the matrices
$$\left(\begin{array}{cc} 1 & 0\\ \phi & 1\\\end{array}\right),\left(\begin{array}{cc} 1 & 1\\ \delta+\phi & 1\\\end{array}\right),\left(\begin{array}{cc} 0 & 1\\ \delta & 0\\\end{array}\right)$$
is not the identity and belongs to $\mathcal{N}_{\delta,\phi}(2)$, and so $\mathcal{N}_{\delta,\phi}(2)\cong\Z/2\Z$ as desired.
\end{proof}

\begin{remark}\label{rem-surjectivity}
    Let $N=p_1^{e_1}p_2^{e_2}\cdots p_r^{e_r}$, where for $1\leq i \leq r$ the $p_i$ are prime and $e_i \in \Z^+$. Then, for each $p_i \mid N$, the projection map $\pi\colon \mathcal{N}_{\delta,\phi}(N) \to \mathcal{N}_{\delta,\phi}(p_i^e)$ is surjective, for any $1\leq e\leq e_i$. Indeed, by the Chinese remainder theorem, we have the following surjective  projection
    $$\mathcal{N}_{\delta,\phi}(N) \ \cong \ \prod_{i=1}^r \mathcal{N}_{\delta,\phi}(p_i^{e_i}) \ \to \ \mathcal{N}_{\delta,\phi}(p_i^{e_i}).$$ Moreover, if $c_{\delta,\phi}(\Bar{a},\Bar{b})\in \mathcal{N}_{\delta,\phi}(p_i^e)$, for some $1\leq e\leq e_i$ and for some $\Bar{a},\Bar{b}\bmod p_i^e$ such that $\det(c_{\delta,\phi}(\Bar{a},\Bar{b})) \not\equiv 0 \bmod p_i$,  then  $c_{\delta,\phi}(a,b)\in \mathcal{N}_{\delta,\phi}(p_i^{e_i})$ for any $a\equiv \Bar{a}$ and $b\equiv \Bar{b}\bmod p_i^e$, since $\det(c_{\delta,\phi}(a,b))\equiv \det(c_{\delta,\phi}(\Bar{a},\Bar{b}))\not\equiv 0 \bmod p_i$. Since $c_1$ maps to $c_1$, it follows that the projection $\mathcal{N}_{\delta,\phi}(p_i^{e_i})\to \mathcal{N}_{\delta,\phi}(p_i^e)$ is surjective as claimed.
\end{remark}

The following result tells us when the largest possible mod-$N$ CM image is abelian. 

\begin{theorem}\label{abelian_n=2}\label{only_abelian}\label{thm-fullnormalizerabelian}
   Let $N\geq 2$. The group $\mathcal{N}_{\delta,\phi}(N)$ is abelian if and only if  $N=2$ and $\phi\equiv 0 \bmod 2$ or $(\delta,\phi) \equiv (0,1) \bmod 2$. Moreover, if $\mathcal{N}_{\delta,\phi}(2)$ is abelian, then it is isomorphic to $\Z/2\Z$ as a group.
\end{theorem}

\begin{proof}
When $N=2$, the group $\mathcal{N}_{\delta,\phi}(N)$ is abelian if and only if $\phi\equiv 0 \bmod 2$ or $(\delta,\phi) \equiv (0,1) \bmod 2$, and if it is abelian, it is isomorphic to $\Z/2\Z$, by Lemma \ref{lem-abelianincase2}. Thus, we may assume $N\geq 3$. 

Let $N=p_1^{e_1}p_2^{e_2}\cdots p_r^{e_r}>2$. Then,  either $N$ is divisible by $4$ or there is a prime $p>2$ such that $N\equiv 0\bmod p$. Further, since the projection maps $\mathcal{N}_{\delta,\phi}(N)\to \mathcal{N}_{\delta,\phi}(p_i^e)$, for any $1\leq e\leq e_i$, are surjective (Remark \ref{rem-surjectivity}), it suffices to show that $\mathcal{N}_{\delta,\phi}(4)$ and $\mathcal{N}_{\delta,\phi}(p_i)$ for $p_i>2$ are non-abelian. Thus, let $N=q$ where $q=4$ or an odd prime $p$.

We have that $c_1\in \mathcal{N}_{\delta,\phi}(q)$. By Lemma \ref{abelian_conditions}, and in order to show that $\mathcal{N}_{\delta,\phi}(q)$ is non-abelian,  it suffices to exhibit an element $c_{\delta,\phi}(a,b)\in \mathcal{N}_{\delta,\phi}(q)$ such that $2b\not\equiv 0 \bmod q$ or $b\phi\not\equiv 0 \bmod q$. 
\begin{itemize}
	\item If $\delta$ is a unit mod $q$, then $c_{\delta,\phi}(0,1)=\left(\begin{array}{cc} \phi & 1\\ \delta & 0\\\end{array}\right)$ has determinant $\det(c_{\delta,\phi}(0,1))\equiv -\delta \bmod q$ (a unit), and therefore $c_{\delta,\phi}(0,1)\in \mathcal{N}_{\delta,\phi}(q)$. In this case $2b\equiv 2\not\equiv 0 \bmod q$.
 
	\item If $\delta$ and $\phi$ are not units mod $q$, then $c_{\delta,\phi}(1,1)=\left(\begin{array}{cc} 1+\phi & 1\\ \delta & 1\\\end{array}\right)$ has determinant $\det(c_{\delta,\phi}(1,1))\equiv 1+\phi-\delta \bmod q$ (a unit), and therefore $c_{\delta,\phi}(1,1)\in \mathcal{N}_{\delta,\phi}(q)$. In this case $2b\equiv 2\not\equiv 0 \bmod q$.
 
\item If $q=4$ and $\phi \equiv 1 \bmod 2$, then $c_{\delta,\phi}(1,2)=\left(\begin{array}{cc} 1+2\phi & 2\\ 2\delta & 1\\\end{array}\right)$ has determinant $\det(c_{\delta,\phi}(1,2))\equiv 1+2\phi\equiv 1\bmod 2$, and therefore $c_{\delta,\phi}(1,2)\in \mathcal{N}_{\delta,\phi}(4)$. In this case $b\phi\equiv 2\not\equiv 0 \bmod 4$.
 
	\item If $q=p$, and if $\delta\equiv 0 \bmod p$ and $\phi \not\equiv 0 \bmod p$, let $b\in (\Z/p\Z)^\times$ such that $b\not\equiv -\phi^{-1} \bmod p$. Notice that $p\geq 3$ so there are $p-2>0$ choices for $b\in (\Z/p\Z)^\times \setminus \{-\phi^{-1}\bmod p\}$. Then, $c_{\delta,\phi}(1,b)=\left(\begin{array}{cc} 1+b\phi & b\\ \delta b & 1\\\end{array}\right)$ has determinant $\det(c_{\delta,\phi}(1,b))\equiv 1+b\phi\not\equiv 0\bmod p$, and therefore $c_{\delta,\phi}(1,b)\in \mathcal{N}_{\delta,\phi}(p)$. In this case $2b\not\equiv 0 \bmod p$.
\end{itemize}
Thus, in all cases, there is $c_1$ and a Cartan element $c_{\delta,\phi}(a,b)\in \mathcal{N}_{\delta,\phi}(q)$ that show (via Lemma \ref{abelian_conditions}) that $\mathcal{N}_{\delta,\phi}(q)$ cannot be abelian for $q=4$ or an odd prime $p$. This concludes the proof of the theorem.
\end{proof}


\section{Abelian and non-abelian division fields in the prime power case}\label{applying_results}

In \cite{lozano-galoiscm}, the second author explicitly describes the groups of $\GL(2, \Z_p)$ that can occur as images of the Galois representation $\rho_{E,p^\infty}\colon \GQ \to \Aut(T_p(E))\cong \GL(2,\Z_p)$, up to conjugation, for an elliptic curve $E/\Q(j_{K,f})$ with CM by an arbitrary order $\mathcal{O}_{K,f}$. In this section, we will apply the results from Section \ref{proofs} to all possible images $G_{E,N} = \operatorname{im} \rho_{E,N}$ from \cite{lozano-galoiscm} and analyze under what circumstances we have that  $G_{E,N}$ is abelian, when $N$ is a prime power. We will deal with general composite values of $N$ in Section \ref{sec-proof-main-theorem1}.

We begin with the case of primes that do not divide the discriminant of the order.

\begin{prop}\label{prop-goodrednprimes}
Let $E/\Q(j_{K,f})$ be an elliptic curve with CM by an order $\Of$ of $K$ of conductor $f\geq 1$, with $j_{K,f}\neq 0$. Let $p$ be a prime that does not divide $2\Delta_Kf$. Then, the image $G_{E,p}$ of $\rho_{E,p}$, and $G_{E,p^\infty}$ are non-abelian. 
\end{prop}
\begin{proof}
    Let $E/\Q(j_{K,f})$ and $p$ be as in the statement. Then, \cite[Theorem 1.2.(4)]{lozano-galoiscm} implies that the image of $\rho_{E,p^\infty}$ is isomorphic to $\mathcal{N}_{\delta,\phi}(\Z_p^\infty)$ and, therefore, the image $G_{E,p}$ of $\rho_{E,p}$ is isomorphic to $\mathcal{N}_{\delta,\phi}(\Z/p\Z)$. Since $p>2$, now Theorem \ref{thm-fullnormalizerabelian} implies that $G_{E,p}$ cannot be abelian, and therefore $G_{E,p^\infty}$ is non-abelian.
\end{proof}

Next we treat the case of elliptic curves with $j(E)=0$ and $p>3$ (i.e., primes $p$ that do not divide $2f\Delta_K=-6$). Before we state and prove the next two propositions we remark that even though we generally define $\delta$ and $\phi$ and choose a basis of $T_p(E)$ such that the image of $\rho_{E,p^\infty}$ is contained in $\mathcal{N}_{\delta,\phi}(\Z_p^\infty)$, in some cases one can choose a different basis so that the image is contained in the simpler form group $\mathcal{N}_{\delta',0}(\Z_p^\infty)$, for some value of $\delta'\in \Z_p$. See the beginning of \cite[Section 5]{lozano-galoiscm} for more details on this.

\begin{prop}\label{prop-jzero-goodredn} 
    Let $E/\Q$ be an elliptic curve with $j(E)=0$ and let $p>3$ be a prime. Then, the image $G_{E,p}$ of $\rho_{E,p}$ is non-abelian. Therefore, the image of $\rho_{E,p^\infty}$ is also non-abelian. 
\end{prop}

\begin{proof} Let $E/\Q$ be an elliptic curve with $j(E)=0$. Thus, $E$ has CM by $\Of$, where 
$K=\Q(\sqrt{-3})$, $f=1$, and $\OO_{K,f} = \OO_K = \Z[(1+\sqrt{-3})/2]$. Let $p>3$ be a prime. Then, the $p$-adic image of $\rho_{E,p^\infty}$ is described by Theorem 1.4 of \cite{lozano-galoiscm}. In particular, the image depends on the class of $p\bmod 9$, so we consider three separate cases, which in turn correspond to parts (1), (2), and (3) of \cite[Thm. 1.4]{lozano-galoiscm}.
\begin{enumerate}
    \item If $p\equiv \pm 1 \bmod 9$, then there is a $\Z_p$-adic basis of $T_p(E)$ such that $G_{E,p}\cong \mathcal{N}_{\delta',0}(p)$ and $G_{E,p^\infty}\cong  \mathcal{N}_{\delta',0}(p^\infty)$, where $\delta' = \Delta_K/4=-3/4$. Since $p>2$, by Theorem \ref{abelian_n=2}, we have that the image $G_{E,p}$ is non-abelian. Since $G_{E,p^\infty}\equiv G_{E,p}\bmod p$ and $G_{E,p}$ is non-abelian, the full $p$-adic image $G_{E,p^\infty}$ is also non-abelian. 

\item If $p\equiv 2$ or $5 \bmod 9$, then either $G_{E,p}\cong \mathcal{N}_{\delta',0}(p)$, or $G_{E,p}\cong \langle \cC_{\delta',0}(p)^3, c_1\rangle$, with $\delta'=-3/4$ as before. In the case that $G_{E,p}\cong \mathcal{N}_{\delta',0}(p)$, the image $G_{E,p}$ is non-abelian as in part (1). If $G_{E,p}\cong \langle \cC_{\delta',0}(p)^3, c_1 \rangle$, consider $c_{\delta',0}(0,1)^3 = c_{\delta',0}(0,\delta')\in \cC_{\delta',0}(p)^3$, where $b \equiv \delta' \bmod p$. Since $b\equiv \delta'\equiv -3/4\in (\Z/p\Z)^\times$, by Corollary \ref{nonabelian}, the image $G_{E,p}$ is non-abelian, and therefore $G_{E,p^\infty}$ is also non-abelian.

\item If $p\equiv 4$ or $7\bmod 9$, then either $G_{E,p^\infty}\cong\mathcal{N}_{\delta',0}(p^\infty)$, or there is a $\Z_p$-basis of $T_p(E)$ such that $G_{E,p^\infty}$ is isomorphic to the subgroup 
\[
H = \left\langle \left\{\left(\begin{array}{cc} a & 0\\ 0 & b\\ \end{array}\right) : a/b \in (\Z_p^\times)^3\right\}, \gamma = \left(\begin{array}{cc} 0 & 1\\ 1 & 0\\ \end{array}\right)\right\rangle.
\]
If $G_{E,p^\infty}=\mathcal{N}_{\delta',0}(p^\infty)$, then image $G_{E,p}$ and $G_{E,p^\infty}$ are non-abelian as before. If $G_{E,p^\infty}\cong H$, we note that 
$$\left(\begin{array}{cc} -1 & 0\\ 0 & 1\\ \end{array}\right)\cdot \left(\begin{array}{cc} 0 & 1\\ 1 & 0\\ \end{array}\right)=\left(\begin{array}{cc} 0 & -1\\ 1 & 0\\ \end{array}\right)\neq \left(\begin{array}{cc} 0 & 1\\ -1 & 0\\ \end{array}\right)=\left(\begin{array}{cc} 0 & 1\\ 1 & 0\\ \end{array}\right)\cdot \left(\begin{array}{cc} -1 & 0\\ 0 & 1\\ \end{array}\right),$$
shows that $H\bmod p$ is non-abelian, since $-1\in (\Z_p^\times)^3$. Hence $G_{E,p}\equiv H\bmod p$ is non-abelian for all $p>2$, and $G_{E,p^\infty}$ is also non-abelian. 
\end{enumerate}
Thus, in all cases we have that $G_{E,p}$ and $G_{E,p^\infty}$ are non-abelian groups.
\end{proof}

The next result describes the possible abelian images for an elliptic curve $E/\Q(j_{K,f})$ with CM by $\mathcal{O}_{K,f}$, for any $f\geq 1$, when $p>2$ divides $f\Delta_K$.

\begin{prop}\label{thm-badprimes-intro}\label{cor-abelianforp3}\label{prop-abelianforp3}
    Let $E/\Q(j_{K,f})$ be an elliptic curve with CM by $\Of$ of conductor $f\geq 1$. Let $p$ be an odd prime dividing $f\Delta_K$ (thus, $j\neq 1728$ where $f\Delta_K=-4$), and choose a $\Z_p$-basis of $T_p(E)$ such that $G_{E,p^\infty}$ is a subgroup of $\mathcal{N}_{\delta',0}(p^\infty)$ with $\delta'=\Delta_Kf^2/4$. Then, $G_{E,p^n}$ for $n\geq 2$, and $G_{E,p^\infty}$ are non-abelian. Moreover,   if $G_{E,p}$ is abelian, then $j_{K,f}=0$ and $p=3$, and we are in one of the two following cases according to the $3$-adic image of $E$:
        \begin{enumerate}
            \item  $G_{E,3^\infty}$ is conjugate to the subgroup generated by $c_{\varepsilon}$ and 
		$$H_1 = \left\{ \left(\begin{array}{cc} a & b\\ -3b/4 & a\\ \end{array}\right): a\in \Z_3^\times,\ b\equiv 0 \bmod 3 \right\}.$$
        In this case, $[\mathcal{N}_{\delta',0}(3^\infty):G_{E,3^\infty}]=3$ and $G_{E,3}\cong (\Z/2\Z)^2$ is abelian, but the image $G_{E,3^n}$ is non-abelian for $n\geq 2$. Moreover, $E/\Q$ can be given by Weierstrass form $y^2=x^3+d$ with $d\in\Z$ such that $d$ and $-3d$ are not squares and $-4d$ is a cube.

            \item  $G_{E,3^\infty}$ is conjugate  to the subgroup generated by $c_{\varepsilon}$ and one of
		$$ H_1' = \left\{ \left(\begin{array}{cc}  a & b\\ -3b/4 & a\\ \end{array}\right) : a\equiv 1,\ b\equiv 0 \bmod 3\Z_3 \right\}.$$
        In this case, $[\mathcal{N}_{\delta',0}(3^\infty):G_{E,3^\infty}]=6$ and $G_{E,3}\cong \Z/2\Z$ is abelian, but the image $G_{E,3^n}$ is non-abelian for $n\geq 2$. Moreover, $E/\Q$ can be given by Weierstrass form $y^2=x^3+d$ with $d\in\Z$ such that $d$ or $-3d$ is a square and $-4d$ is a cube.
        \end{enumerate}
\end{prop}

\begin{proof}
 Let $E/\Q(j_{K,f})$ be an elliptic curve with CM by $\Of$ of conductor $f\geq 1$. Let $p$ be an odd prime dividing $f\Delta_K$ (this assumption automatical excludes $j_{K,f}=1728$). The image of $\rho_{E,p^\infty}$ is given in this case by \cite[Theorem 1.5]{lozano-galoiscm}. There are two cases to consider according to whether $j_{K,f}\neq 0$ or $j_{K,f}=0$, which correspond accordingly to cases (a) and (b) of the statement of \cite[Thm. 1.5]{lozano-galoiscm}.
\begin{enumerate}
    \item[(a)] If $j_{K,f}\neq 0,1728$, then either $G_{E,p^\infty} \cong  \mathcal{N}_{\delta',0}(p^\infty)$ or $G_{E,p^\infty}$ is isomorphic to the group of $\GL(2,\Z_p)$ generated by $c_\varepsilon$ and 
\[
H = \left\{ \left(\begin{array}{cc} a & b\\ \delta b & a\\ \end{array}\right): a\in {\Z_p^\times}^2, b\in  \Z_p\right\},
\]
where $\delta' = \Delta_Kf^2/4\equiv 0 \bmod p$. 
In the case that $G_{E,p^\infty} \cong \mathcal{N}_{\delta',0}(p^\infty)$, by Theorem \ref{abelian_n=2}, since $p>2$, we have that the image $G_{E,p}$ is non-abelian, and hence, the image $G_{E,p^\infty}$ is also non-abelian. Now let us consider the case that $G_{E,p^\infty}\cong  \langle c_\varepsilon, H\rangle$. Observe that $c_{\delta',0}(1,1) \in G_{E,p}$ (its determinant is a unit congruent to $1^2\equiv 1 \bmod p$). Since $b\equiv 1 \in (\Z/p\Z)^\times$, Corollary \ref{nonabelian} implies that the image $G_{E,p}$ is non-abelian, and hence, the image $G_{E,p^\infty}$ is also non-abelian. 

\item[(b)] Now let $j_{K,f}=0$, and so $\Delta_Kf^2 = -3$. Then, $p=3$ and either $G_{E,3^\infty} \cong \mathcal{N}_{\delta',0}(3^\infty)$ or $[\mathcal{N}_{\delta',0}(3^\infty):G_{E,3^\infty}]=2,3,$ or $6$ and one of the upcoming (i), (ii), or (iii) hold. In the case that $G_{E,3^\infty} \cong \mathcal{N}_{\delta',0}(3^\infty)$, the  images $G_{E,3}$ and  $G_{E,3^\infty}$ are non-abelian as before. If that is not the case, then $[\mathcal{N}_{\delta',0}(3^\infty):G_{E,3^\infty}]=2,3,$ or $6$, which we consider as separate cases.
\begin{enumerate}
    \item[(i)] If $[\mathcal{N}_{\delta',0}(3^\infty):G_{E,3^\infty}]=2$, then $G_{E,3^\infty}$ is generated by $c_\varepsilon$ and \[\left\{ \left(\begin{array}{cc}  a & b\\ -3b/4 & a\\ \end{array}\right): a,b\in  \Z_3,\ a\equiv 1 \bmod 3\right\}.\] Let $a,b \equiv 1 \bmod 3$. Observe that $c_{\delta',0}(1,1) \in G_{E,3}$. Since $b \in (\Z/3\Z)^\times$, by Corollary \ref{nonabelian}, we have that $G_{E,3}$ and  $G_{E,3^\infty}$ are non-abelian. 

    \item[(ii)] If $[\mathcal{N}_{\delta',0}(3^\infty):G_{E,3^\infty}]=3$, then we have the following three cases:
    \begin{itemize}
        \item[(1)] $G_{E,3^\infty}$ is generated by $c_\varepsilon$ and 
        \[
        H_1 = \left\{ \left(\begin{array}{cc} a & b\\ -3b/4 & a\\ \end{array}\right): a\in \Z_3^\times,\ b\equiv 0 \bmod 3 \right\}.
        \]
        Observe that $G_{E,3}$ is generated by the diagonal matrices (i.e., $H_1\bmod 3$) and $c_\varepsilon$, which commute. Therefore, the image $G_{E,3}$ is abelian in this case and $G_{E,3}\cong \Z/2\Z\oplus \Z/2\Z$. However, we will show that the image $G_{E,9}$ with index 3 is non-abelian. Let $a\equiv 1 \bmod 9$ and $b \equiv 6 \bmod 9$, so $a \in (\Z/3\Z)^\times$ and $b \equiv 0 \bmod 3$. Observe that $c_{\delta',0}(1,6) \in G_{E,9}$ since its determinant is congruent to $1 \bmod 3$, and $2b \equiv 12 \not\equiv 0 \bmod 9$. Thus, by Lemma \ref{abelian_conditions}, we have that the image $G_{E,9}$ is non-abelian. Therefore, $G_{E,3^\infty} = \langle c_\varepsilon, H_1\rangle$ is non-abelian, and the image $G_{E,3^n}\cong  \langle c_\varepsilon, H_1\rangle$ is non-abelian for $n\geq 2$.

        \item[(2)] $G_{E,3^\infty}$ is generated by $c_\varepsilon$ and 
        \[
        H_2 = \left\langle \left(\begin{array}{cc}  2 & 0\\ 0 & 2\\ \end{array}\right) ,\left(\begin{array}{cc}  1 & 1\\ -3/4 & 1\\ \end{array}\right)\right\rangle.
        \]
        In this case, observe that $c_{\delta',0}(1,1) \in  G_{E,3}$. Since $b \in (\Z/3\Z)^\times$, by Corollary \ref{nonabelian}, we have that the image $G_{E,3}$ is non-abelian. Therefore, the image $G_{E,3^\infty}= \langle c_\varepsilon, H_2\rangle$ with index 3 is non-abelian. 

        \item[(3)] $G_{E,3^\infty}$ is generated by $c_\varepsilon$ and 
        \[
        H_3 = \left\langle \left(\begin{array}{cc}  2 & 0\\ 0 & 2\\ \end{array}\right) ,\left(\begin{array}{cc}  -5/4 & 1/2\\ -3/8 & -5/4\\ \end{array}\right)\right\rangle.
        \]
        In this case, observe that $c_{\delta',0}(-5/4,1/2) \in G_{E,3}$. Since $b\equiv 1/2 \equiv 2 \in (\Z/3\Z)^\times$, by Corollary \ref{nonabelian}, we have that the image $G_{E,3}$ is non-abelian. Therefore, the image $G_{E,3^\infty} = \langle c_\varepsilon, H_3\rangle$ with index 3 is non-abelian.
    \end{itemize}

    \item[(iii)] If $[\mathcal{N}_{\delta',0}(3^\infty):G_{E,3^\infty}]=6$, then we have the following three cases:
    \begin{itemize}
        \item[(1)] $G_{E,3^\infty}$ is generated by $c_\varepsilon$ and 
        \[
        H_1' = \left\{ \left(\begin{array}{cc}  a & b\\ -3b/4 & a\\ \end{array}\right) : a\equiv 1,\ b\equiv 0 \bmod 3\Z_3 \right\}.
        \]
        Observe that $H_1' \equiv \{\operatorname{Id}\} \bmod 3$, so $G_{E,3}\equiv \langle c_\varepsilon,H_1'\rangle \equiv \{\operatorname{Id},c_\varepsilon\}\bmod 3$ is abelian and isomorphic to $\Z/2\Z$. However, we claim that the image $G_{E,9}$ is non-abelian. Let $a\equiv 1 \bmod 9$ and $b \equiv 6 \bmod 9$, so $a \in (\Z/3\Z)^\times$ and $b \equiv 0 \bmod 3$. Observe that $c_{\delta',0}(1,6) \in G_{E,9}$ and $2b \equiv 12 \not\equiv 0 \bmod 9$. Thus, by Lemma \ref{abelian_conditions}, we have that the image $G_{E,9}$ is non-abelian. Therefore, $G_{E,3^\infty} = \langle c_\varepsilon, H_1'\rangle$ is non-abelian, and the image $G_{E,3^n}\cong  \langle c_\varepsilon, H_1'\rangle$ is non-abelian for $n\geq 2$.

        \item[(2)] $G_{E,3^\infty}$ is generated by $c_\varepsilon$ and 
        \[
        H_2' = \left\langle \left(\begin{array}{cc}  4 & 0\\ 0 & 4\\ \end{array}\right) ,\left(\begin{array}{cc}  1 & 1\\ -3/4 & 1\\ \end{array}\right)\right\rangle.
        \]
        In this case, observe that $c_{\delta',0}(1,1) \in G_{E,3}$. Since $b \in (\Z/3\Z)^\times$, by Corollary \ref{nonabelian}, we have that the image $G_{E,3}$ is non-abelian. Therefore, the image $G_{E,3^\infty} = \langle c_\varepsilon, H_2'\rangle$ with index 6 is non-abelian.

        \item[(3)] $G_{E,3^\infty}$ is generated by $c_\varepsilon$ and 
        \[
        H_3' = \left\langle \left(\begin{array}{cc}  4 & 0\\ 0 & 4\\ \end{array}\right) ,\left(\begin{array}{cc}  -5/4 & 1/2\\ -3/8 & -5/4\\ \end{array}\right)\right\rangle.
        \]
        In this case, observe that $c_{\delta',0}(-5/4,1/2) \in G_{E,3}$. Since $b\equiv 1/2\equiv 2 \in (\Z/3\Z)^\times$, by Corollary \ref{nonabelian}, we have that the image $G_{E,3}$ is non-abelian. Therefore, the image $G_{E,3^\infty} = \langle c_\varepsilon, H_3'\rangle$ with index 6 is non-abelian.
    \end{itemize}
\end{enumerate}
\end{enumerate}
Thus in all cases $G_{E,p^\infty}$ and $G_{E,p^n}$ for $n\geq 2$ are non-abelian as claimed, and the only cases where $G_{E,p}$ is abelian are as described in the statement. The Weierstrass forms given in the statement follow from Proposition 1.16 in \cite{zywina}. Indeed, Zywina shows that there are $5$ possible images modulo $3$ for an elliptic curve $E/\Q$ with $j(E)=0$, namely $B(3)$, $H_{3,1}$, $H_{3,2}$, $H_{1,1}$, or $G_1$, in the notation of \cite[Section 1.2]{zywina}. Out of these, it is easy to check that only $G_1$ and $H_{1,1}$ are abelian, of index $3$ and $6$, respectively. This concludes the proof.
\end{proof}

Before we move on to study the mod-$2^n$ and $2$-adic images of CM curves, we write down a corollary for $N=2$ that follows from our previous Lemma \ref{lem-abelianincase2} and Theorem \ref{abelian_n=2}, which established that $\mathcal{N}_{\delta,\phi}(2)$ is abelian if and only if $\phi\equiv 0 \bmod 2$ or $(\delta,\phi) \equiv (0,1) \bmod 2$.

\begin{prop}\label{prop-case_of_n_eq_2}
    Let $E/\Q(j_{K,f})$ be an elliptic curve with CM by an order $\Of$ in an imaginary quadratic field $K$, and let $G_{E,2}=\operatorname{Im}(\rho_{E,2})$. 
    \begin{enumerate}
        \item If $G_{E,2}\subsetneq \mathcal{N}_{\delta,\phi}(2)$, then $G_{E,2}$ is abelian.
        \item If $G_{E,2}\cong \mathcal{N}_{\delta,\phi}(2)$, then $G_{E,2}$ is abelian if and only if $G_{E,2}\cong \Z/2\Z$ and either
        \begin{enumerate}
            \item $\Delta_Kf^2\equiv 0 \bmod 4$, or 
            \item  $\Delta_K\equiv 1 \bmod 8$ and $f\equiv 1 \bmod 2$.
        \end{enumerate}
    \end{enumerate}
\end{prop}
\begin{proof} Notice that $G_{E,2}\subseteq \mathcal{N}_{\delta,\phi}(2)$ which is in turn a subgroup of $\GL(2,\Z/2\Z)\cong S_3$. Thus, any proper subgroup of $\mathcal{N}_{\delta,\phi}(2)$ must be abelian. On the other hand, if $G_{E,2}\cong \mathcal{N}_{\delta,\phi}(2)$, then Theorem \ref{abelian_n=2} says that it is abelian if and only if $\phi\equiv 0 \bmod 2$ or $(\delta,\phi) \equiv (0,1) \bmod 2$. Recall that
    	\begin{itemize}
\item If $\Delta_Kf^2\equiv 0\bmod 4$, let $\delta=\Delta_K f^2/4$, and $\phi=0$.
		\item If $\Delta_Kf^2\equiv 1 \bmod 4$, let $\delta=\frac{(\Delta_K-1)}{4}f^2$, let $\phi=f$.
	\end{itemize}
 If $\phi\equiv 0 \bmod 2$, then $\Delta_Kf^2\equiv 0 \bmod 4$, and therefore $\phi=0$. If $\delta\equiv 0 \bmod 2$ and $\phi\equiv 1 \bmod 2$, then we must have $f\equiv 1 \bmod 2$ and $(\Delta_K-1)/4\equiv 0 \bmod 2$, and therefore $\Delta_K\equiv 1 \bmod 8$. This concludes the proof.
\end{proof}

The following proposition treats the case of the $2$-adic Galois representation attached to an elliptic curve $E/\Q(j_{K,f})$ with CM by $\mathcal{O}_{K,f}$, with $j_{K,f}\neq 0,1728$.

\begin{prop}\label{thm-2adictwist-intro}\label{prop-jnot01728andell2}
    Let $E/\Q(j_{K,f})$ be an elliptic curve with CM by an order $\Of$ in an imaginary quadratic field $K$, with $j_{K,f}\neq 0$ or $1728$. 
    Then, the image $G_{E,2^n}$ is non-abelian for $n\geq 2$. Finally, $G_{E,2}$ is abelian if and only if $G_{E,2}\cong \Z/2\Z$ and either
      \begin{enumerate}
            \item $\Delta_Kf^2\equiv 0 \bmod 4$, or 
            \item  $\Delta_K\equiv 1 \bmod 8$ and $f\equiv 1 \bmod 2$.
        \end{enumerate}
\end{prop}

\begin{proof}
Let $E/\Q(j_{K,f})$ be an elliptic curve with CM by an order $\Of$ of conductor $f\geq 1$. Let $p=2$ and $j_{K,f}\neq 0$ or $1728$. The image of $\rho_{E,2^\infty}$ is given in this case by \cite[Theorem 1.6]{lozano-galoiscm}. There are two cases to consider.

First, if $G_{E,2^\infty}\cong \mathcal{N}_{\delta,\phi}(\Z_2)$, then Theorem \ref{abelian_n=2} shows that $G_{E,2^\infty}$ and $G_{E,2^n}$ are non-abelian for all $n\geq 2$. And Prop. \ref{prop-case_of_n_eq_2} shows that $G_{E,2}$ will be abelian if and only if $\Delta_Kf^2\equiv 0\bmod 4$ or $\Delta_kf^2\equiv 1 \bmod 4$ and $\Delta_K\equiv 1 \bmod 8$.

Second, if $G_{E,2^\infty}\subsetneq \mathcal{N}_{\delta,\phi}(\Z_2)$, then \cite[Thm. 1.6]{lozano-galoiscm} says that for $\alpha \in \{3,5\}$, the image $G_{E,2^\infty}$ is a conjugate of the topological closures of
\[
H_1 = \left\langle \left(\begin{array}{cc} \varepsilon & 0\\ 0 & -\varepsilon\\\end{array}\right),\left(\begin{array}{cc} \alpha & 0\\ 0  & \alpha\\\end{array}\right),\left(\begin{array}{cc} 1 & 1\\ \delta  & 1\\\end{array}\right)\right\rangle \text{ or } H_2 = \left\langle \left(\begin{array}{cc} \varepsilon & 0\\ 0 & -\varepsilon\\\end{array}\right),\left(\begin{array}{cc} \alpha & 0\\ 0  & \alpha\\\end{array}\right),\left(\begin{array}{cc} -1 & -1\\ -\delta  & -1\\\end{array}\right)\right\rangle,
\]
where $H_1,H_2 \subseteq \GL(2,\Z_2)$. Moreover, either $\Delta_k f^2 \equiv 0 \bmod 16$ or $\Delta_K\equiv 0 \bmod 8$ (so in both cases $\Delta_K f^2\equiv 0 \bmod 4$). In either case, $\phi=0$ and $\delta=\Delta_K f^2/4 \equiv 0 \bmod 2$. Thus, $H_i \bmod 2$ reduces to 
\[
\left\langle \left(\begin{array}{cc} 1 & 1\\ 0  & 1\\\end{array}\right)\right\rangle \subseteq \GL(2,\Z/2\Z) \]
which is abelian and isomorphic to $\Z/2\Z$.

Next we shall show that $G_{E,4}$ is non-abelian. First we will look at $G_{E,2^\infty} = H_1$. Consider $c_{\delta,0}(1,1) \in H_1$ mod $4$. Since $b\in (\Z/4\Z)^\times$, by Corollary \ref{nonabelian}, we have that the image $G_{E,4}$ is non-abelian. Therefore, the image $G_{E,2^n}$ is non-abelian for $n\geq 2$. 

Now we will look at $G_{E,2^\infty} = H_2$. Consider $c_{\delta,0}(-1,-1) \in H_2 \bmod 4$. Since $b\in (\Z/4\Z)^\times$, by Corollary \ref{nonabelian}, we have that the image $G_{E,4}$ is non-abelian. Therefore, the image $G_{E,2^n}$ is non-abelian for $n\geq 2$.
\end{proof}

The following proposition will treat the case of the $2$-adic Galois representation attached to an elliptic curve $E/\Q(j_{K,f})$ with CM by $\mathcal{O}_{K,f}$, with $j_{K,f}=1728$. But, we first need a lemma.

\begin{lemma}\label{abelian_conditions_j1728}
     Let $N\geq 2$ and let  $G\subseteq \mathcal{N}_{\delta,0}(N)$ be a subgroup (so here $\phi=0$). Let
     \[
    \gamma' \in \left\{ c_1=\left(\begin{array}{cc} 1 & 0\\ 0 & -1\\\end{array}\right),c_{-1}=\left(\begin{array}{cc} -1 & 0\\ 0 & 1\\\end{array}\right),c_1'=\left(\begin{array}{cc} 0 & 1\\ 1 & 0\\\end{array}\right),c_{-1}'=\left(\begin{array}{cc} 0 & -1\\ -1 & 0\\\end{array}\right) \right\} = \Gamma', 
    \]
     If $\gamma', c_{\delta,0}(a,b) \in G$, for some $a,b\in \Z/N\Z$, such that $\gamma' \cdot c_{\delta,0}(a,b) = c_{\delta,0}(a,b) \cdot \gamma'$, then $2b\equiv 0 \bmod N$ or $b(\delta-1)\equiv 0 \bmod N$.  
     \end{lemma}

\begin{proof}
    Suppose that $\gamma'$ and $c_{\delta,0}(a,b)\in G$, where $a,b\in \mathbb{Z}/N\mathbb{Z}$ and $a^2+ab\phi-\delta b^2 \in (\mathbb{Z}/N\mathbb{Z})^\times$, and suppose the two matrices commute. If $\gamma'=c_\varepsilon$ for $\varepsilon\in \{ \pm 1\}$ then Lemma \ref{abelian_conditions} shows that $2b\equiv 0 \bmod N$. If $\gamma'=c_\varepsilon'$, then
     \begin{align*}
        \left(\begin{array}{cc} a & b \\ \delta b & a\\\end{array}\right)\left(\begin{array}{cc} 0 & \varepsilon \\ \varepsilon & 0\\\end{array}\right) &\equiv \left(\begin{array}{cc} 0 & \varepsilon \\ \varepsilon & 0\\\end{array}\right)\left(\begin{array}{cc} a & b \\ \delta b & a\\\end{array}\right) \bmod N, \text{ and therefore} \\
        \left(\begin{array}{cc} \varepsilon b & \varepsilon a  \\ \varepsilon a & \varepsilon\delta b\\\end{array}\right) &\equiv \left(\begin{array}{cc} \varepsilon\delta b & \varepsilon a\\ \varepsilon a & \varepsilon b\\\end{array}\right) \bmod N,
    \end{align*}
    In order for the last equivalence to hold, we must have that $b(\delta-1)\equiv 0 \bmod N$. 
    Thus, either $2b\equiv 0 \bmod N$ or $b(\delta-1) \equiv 0 \bmod N$, as desired.
\end{proof}

The following proposition uses notation from \cite[Theorem 1.7]{lozano-galoiscm} for the classification of images. The images themselves will be spelled out in the proof.

\begin{prop}\label{thm-j1728-intro}\label{prop-j1728andell2}
    Let $E/\Q$ be an elliptic curve with $j(E)=1728$, and let $c\in \GQ$ be a complex conjugation, and $\gamma=\rho_{E,2^\infty}(c)$. Let $G_{E,2^\infty}$ be the image of $\rho_{E,2^\infty}$ and let $G_{E,K,2^\infty}=\rho_{E,2^\infty}(G_{\Q(i)})$.
    \begin{enumerate}
        \item[(i)] If $[\cC_{-1,0}(2^\infty):G_{E,K,2^\infty}]=1$, then the image $G_{E,2^n}$ is non-abelian if and only if $n\geq 2$. Here $G_{E,2}\cong \Z/2\Z$.

        \item[(ii)] If $[\cC_{-1,0}(2^\infty):G_{E,K,2^\infty}]=2$, then either $G_{E,2^\infty} = \langle \gamma, G_{2,a}\rangle$ or $\langle \gamma, G_{2,b}\rangle$:
        \begin{itemize}
            \item If $G_{E,2^\infty} = \langle \gamma, G_{2,a}\rangle$, then the image $G_{E,2^n}$ is non-abelian if and only if $n\geq 3$. Here $G_{E,2}\cong \{\operatorname{Id}\}$ or $\Z/2\Z$, and $G_{E,4}\cong (\Z/2\Z)^3$.

            \item If $G_{E,2^\infty} = \langle \gamma, G_{2,b}\rangle$, then the image $G_{E,2^n}$ is non-abelian if and only if $n\geq 2$. Here $G_{E,2}\cong \Z/2\Z$.
        \end{itemize}

        \item[(iii)] If $[\cC_{-1,0}(2^\infty):G_{E,K,2^\infty}]=4$, then either $G_{E,2^\infty} = \langle \gamma, G_{4,a}\rangle$, $\langle \gamma, G_{4,b}\rangle$, $\langle \gamma, G_{4,c}\rangle$, or $\langle \gamma, G_{4,d}\rangle$:
        \begin{itemize}
            \item If $G_{E,2^\infty} = \langle \gamma, G_{4,a}\rangle$ or $\langle \gamma, G_{4,b}\rangle$, then $G_{E,2^n}$ is non-abelian if and only if $n\geq 3$. Here $G_{E,2} \cong \{\operatorname{Id}\}$ or $\Z/2\Z$, and $G_{E,4}\cong (\Z/2\Z)^2$

            \item If $G_{E,2^\infty} = \langle \gamma, G_{4,c}\rangle$ or $\langle \gamma, G_{4,d}\rangle$, then $G_{E,2^n}$ is non-abelian if and only if $n\geq 2$. Here $G_{E,2} \cong \Z/2\Z$.
        \end{itemize}
    \end{enumerate}
\end{prop}

\begin{proof}
    Let $E/\Q$ be an elliptic curve with $j(E)=1728$. Thus, $E$ has CM by $\Of=\Z[i]$, where $K=\Q(i)$. Since $\Delta_K=-4$ and $f=1$, we set $\delta =-1$ and $\phi=0$. Let $c\in \GQ$ be a complex conjugation, and $\gamma=\rho_{E,2^\infty}(c)$. The image of $\rho_{E,2^\infty}$ is given in this case by \cite[Theorem 1.7]{lozano-galoiscm} and tells us that the image $G_{E,2^\infty} = \langle c_{\gamma'}, G_{E,K,2^\infty}\rangle$ where 
    \[
    \gamma' \in \left\{ c_1=\left(\begin{array}{cc} 1 & 0\\ 0 & -1\\\end{array}\right),c_{-1}=\left(\begin{array}{cc} -1 & 0\\ 0 & 1\\\end{array}\right),c_1'=\left(\begin{array}{cc} 0 & 1\\ 1 & 0\\\end{array}\right),c_{-1}'=\left(\begin{array}{cc} 0 & -1\\ -1 & 0\\\end{array}\right) \right\} = \Gamma', 
    \]
    such that  
	$\gamma \equiv \gamma' \bmod 4,$ and $G_{E,K,2^\infty}=\rho_{E,2^\infty}(G_{\Q(i)})$.
 
 There are three cases to consider according to the index of $G_{E,K,2^\infty}$ in $\cC_{-1,0}(2^\infty)$, which correspond accordingly to the cases of the statement of \cite[Theorem 1.7]{lozano-galoiscm}.  


     \begin{enumerate}
         \item[(i)] If $[\cC_{-1,0}(2^\infty):G_{E,K,2^\infty}]=1$, then $G_{E,2^\infty}= \mathcal{N}_{-1,0}(2^\infty)= \langle \gamma', G_1\rangle$, where
         \[
         G_1= \left\{\left(\begin{array}{cc} a & b\\ -b & a\\\end{array}\right)\in \GL(2,\Z_2) : a^2+b^2\not\equiv 0 \bmod 2 \right\}.
         \]
         By Theorem \ref{abelian_n=2} and since $\phi=0$, we have that $G_{E,2}\cong \Z/2\Z$ is abelian and the image $G_{E,2^n}$ is non-abelian for $n\geq 2$.
    \end{enumerate}

    For parts (ii) and (iii), there will be cases where we need to check if the image $G_{E,8}$ is abelian. To check if the image $G_{E,8}$ is abelian, we need to first lift the four matrices in $\Gamma'$ mod 8. Let $\pi_8\colon \mathcal{N}_{-1,0}(8) \to \mathcal{N}_{-1,0}(4)$. We want to consider the images $G_{E,8} = \langle \gamma'', G_{E,K,8}\rangle$, where $\gamma'' \in \pi_8^{-1}(\Gamma')=\Gamma''$. Let $\Gamma'' =\pi_8^{-1}(c_1)\cup \pi_8^{-1}(c_{-1})\cup \pi_8^{-1}(c_1')\cup \pi_8^{-1}(c_{-1}')$. With a little computation, one can show that
    \begin{align*}
        \pi_8^{-1}(c_1) &= \left\{ \left(\begin{array}{cc} 5 & 4\\ 4 & 3\\\end{array}\right), \left(\begin{array}{cc} 1 & 4\\ 4 & 7\\\end{array}\right), \left(\begin{array}{cc} 5 & 0\\ 0 & 3\\\end{array}\right), \left(\begin{array}{cc} 1 & 0\\ 0 & 7\\\end{array}\right) \right\}, \\
        \pi_8^{-1}(c_{-1}) &= \left\{ \left(\begin{array}{cc} 7 & 4\\ 4 & 1\\\end{array}\right), \left(\begin{array}{cc} 7 & 0\\ 0 & 1\\\end{array}\right), \left(\begin{array}{cc} 3 & 0\\ 0 & 5\\\end{array}\right), \left(\begin{array}{cc} 3 & 4\\ 4 & 5\\\end{array}\right) \right\}, \\
        \pi_8^{-1}(c_1') &= \left\{ \left(\begin{array}{cc} 4 & 1\\ 1 & 4\\\end{array}\right), \left(\begin{array}{cc} 0 & 1\\ 1 & 0\\\end{array}\right), \left(\begin{array}{cc} 4 & 5\\ 5 & 4\\\end{array}\right), \left(\begin{array}{cc} 0 & 5\\ 5 & 0\\\end{array}\right) \right\},\\
        \pi_8^{-1}(c_{-1}') &= \left\{ \left(\begin{array}{cc} 4 & 7\\ 7 & 4\\\end{array}\right), \left(\begin{array}{cc} 0 & 7\\ 7 & 0\\\end{array}\right), \left(\begin{array}{cc} 0 & 3\\ 3 & 0\\\end{array}\right), \left(\begin{array}{cc} 4 & 3\\ 3 & 4\\\end{array}\right) \right\}.
    \end{align*}

    \begin{enumerate}
         \item[(ii)] If $[\cC_{-1,0}(2^\infty):G_{E,K,2^\infty}]=2$, then either $G_{E,2^\infty} = \langle \gamma', G_{2,a}\rangle$ or $G_{E,2^\infty} = \langle \gamma', G_{2,b}\rangle$, where 
         \[
         G_{2,a}=\left\langle -\operatorname{Id}, 3\cdot \operatorname{Id},\left(\begin{array}{cc} 1 & 2\\ -2 & 1\\\end{array}\right) \right\rangle \ \text{ and } \ G_{2,b}=\left\langle -\operatorname{Id}, 3\cdot \operatorname{Id},\left(\begin{array}{cc} 2 & 1\\ -1 & 2\\\end{array}\right) \right\rangle.
         \]
         \begin{itemize}
             \item Consider the case $G_{E,2^\infty} = \langle \gamma', G_{2,a}\rangle$. Observe that $G_{2,a}$ reduces to the identity modulo $2$, thus  $G_{E,2} = \langle \gamma', G_{2,a}\rangle\cong \{\operatorname{Id}\}$ or $\Z/2\Z$ (the isomorphism type depends on the choice of $\gamma'$) and it is abelian. For $N=4$, we have that $G_{2,a}=\langle -\operatorname{Id},c_{-1,0}(1,2)\rangle$, and all of these matrices commute with any of the choices for $\gamma'\in \Gamma'$. Thus, we have that the image $G_{E,4}$ is abelian. Moreover, one can easily check that $G_{E,4}\cong (\Z/2\Z)^3$ in all cases.
             
             Finally, we show that $G_{E,8}$ is non-abelian. Let $\gamma'' \in \Gamma''$, so $\gamma'' \equiv \gamma' \bmod 4$. Consider the image $G_{E,8}=\langle \gamma'', G_{2,a}\rangle$. A finite Magma 
 (\cite{magma}) computation shows that the image $G_{E,8}$ is non-abelian. Therefore, the image $G_{E,2^n}$ is non-abelian for $n\geq 3$.

             \item Consider the case $G_{E,2^\infty} = \langle \gamma', G_{2,b}\rangle$. Here $G_{E,2}\equiv \{\operatorname{Id},c_{-1,0}(0,1)\}\cong \Z/2\Z$ is abelian. If we consider $c_{-1,0}(2,1)\bmod 4$, then $b\equiv 1 \bmod 4$. Since $b\in (\Z/4\Z)^\times$, by Lemma \ref{abelian_conditions_j1728}, the image $G_{E,4}$ is non-abelian. Therefore, the image $G_{E,2^n}$ is non-abelian for $n\geq 2$.
         \end{itemize}

         \item[(iii)] If $[\cC_{-1,0}(2^\infty):G_{E,K,2^\infty}]=4$, then either $G_{E,2^\infty} = \langle \gamma', G_{4,a}\rangle$, $G_{E,2^\infty} = \langle \gamma', G_{4,b}\rangle$, $G_{E,2^\infty} = \langle \gamma', G_{4,c}\rangle$, or $G_{E,2^\infty} = \langle \gamma', G_{4,d}\rangle$, where
         \[
         G_{4,a}=\left\langle 5\cdot \operatorname{Id},\left(\begin{array}{cc} 1 & 2\\ -2 & 1\\\end{array}\right) \right\rangle, \  G_{4,b}=\left\langle 5\cdot \operatorname{Id},\left(\begin{array}{cc} -1 & -2\\ 2 & -1\\\end{array}\right) \right\rangle,
         \]
	   \[
         G_{4,c}= \left\langle -3\cdot \operatorname{Id},\left(\begin{array}{cc} 2 & -1\\ 1 & 2\\\end{array}\right) \right\rangle, \ \text{ and } \ G_{4,d}=\left\langle -3\cdot \operatorname{Id},\left(\begin{array}{cc} -2 & 1\\ -1 & -2\\\end{array}\right) \right\rangle.
         \]
         \begin{itemize}
             \item Consider the case $G_{E,2^\infty} = \langle \gamma', G_{4,a}\rangle$. Observe that $G_{4,a}$ reduces to the identity modulo $2$, thus  $G_{E,2} = \langle \gamma', G_{2,a}\rangle\cong \{\operatorname{Id}\}$ or $\Z/2\Z$ (the isomorphism type depends on the choice of $\gamma'$), and it is abelian. For $N=4$, we have that $G_{4,a}=\langle c_{-1,0}(1,2)\rangle$, and this matrix commutes with any of the choices for $\gamma'\in \Gamma'$. Thus, we have that the image $G_{E,4}$ is abelian. Moreover, one can easily check that $G_{E,4}\cong (\Z/2\Z)^2$ in all cases.

             Finally, we show that $G_{E,8}$ is non-abelian. Let $\gamma'' \in \Gamma''$, so $\gamma'' \equiv \gamma' \bmod 4$. Consider the image $G_{E,8}=\langle \gamma'', G_{2,a}\rangle$. A finite Magma computation shows that the image $G_{E,8}$ is non-abelian. Therefore, the image $G_{E,2^n}$ is non-abelian for $n\geq 3$. 

             \item Consider the case $G_{E,2^\infty} = \langle \gamma', G_{4,b}\rangle$. The same arguments as in the previous bullet point shows that $G_{E,2} \cong \{\operatorname{Id}\}$ or $\Z/2\Z$, and $G_{E,4}\cong (\Z/2\Z)^2$, but $G_{E,2^n}$ is non-abelian for $n\geq 3$. 

             \item Consider the case $G_{E,2^\infty} = \langle \gamma', G_{4,c}\rangle$. Here $G_{E,2}\equiv \{\operatorname{Id},c_{-1,0}(0,1)\}\cong \Z/2\Z$ is abelian.
             
             However, if we consider $c_{-1,0}(2,-1) \bmod 4$, then $b\equiv -1 \bmod 4$. Since $b\in (\Z/4\Z)^\times$, by Lemma \ref{abelian_conditions_j1728}, the image $G_{E,4}$ is non-abelian. Therefore, the image $G_{E,2^n}$ is non-abelian for $n\geq 2$.

             \item Consider the case $G_{E,2^\infty} = \langle \gamma', G_{4,d}\rangle$. The same arguments as in the previous bullet point shows that $G_{E,2} \cong \Z/2\Z$, but $G_{E,2^n}$ is non-abelian for $n\geq 2$.
         \end{itemize}
     \end{enumerate}
This concludes the proof.
\end{proof}

The following lemma determines the models of elliptic curves with $j_{K,f}=1728$ that have abelian $4$-division field. We first cite a result that we will use in the proof.

\begin{cor}[\cite{gonzalez-jimenez-lozano-robledo}, Cor. 2.5]\label{cor-E4}
 Let $E/F$ be an elliptic curve defined over a number field $F$, given by 
$$E:y^2 = (x-\alpha)(x-\beta)(x-\gamma)$$
with $\alpha,\beta,\gamma\in \overline{F}$. Then,
$$F(E[4])=F\left(\sqrt{\pm(\alpha-\beta)},\sqrt{\pm(\alpha-\gamma)},\sqrt{\pm(\beta-\gamma)}\right)=F\left(\sqrt{-1},\sqrt{\alpha-\beta},\sqrt{\alpha-\gamma},\sqrt{\beta-\gamma}\right).$$
\end{cor}

Now we are ready for the lemma, which was already proved as part of the proof of \cite[Theorem 4.3]{gonzalez-jimenez-lozano-robledo}, but we include it here for completeness. 

\begin{lemma}\label{lem-4abelianj1728}
 Let $E/\Q$ be an elliptic curve with $j(E)=j_{K,f}=1728$ given by $y^2=x^3+sx$, such that $s\in\Z$ is a $4$-th-power-free integer. Then, $\Q(E[4])/\Q$ is abelian if and only if $s=\pm t^2$ for some other square-free integer $t$. Moreover,
 \begin{enumerate}
 \item $\Q(E[4])=\Q(\zeta_8)$ and $\Gal(\Q(E[4])/\Q)\cong (\Z/2\Z)^2$ if and only if $s\in \{\pm 1,\pm 4\}$, and
 \item $\Q(E[4])=\Q(\zeta_8,\sqrt{t})$ and $\Gal(\Q(E[4])/\Q)\cong (\Z/2\Z)^3$ if and only if $s=\pm t^2$ for some square-free integer $t\not\in \{\pm 1,\pm 2\}$.
 \end{enumerate}
\end{lemma}
\begin{proof}
    If $j(E)=1728$, then $E$ has a model $y^2=x^3+sx$, for some $s\in \Z$ which can be chosen (and we choose) to be $4$-th-power free. Corollary \ref{cor-E4} implies that 
$$\Q(E[4])=\Q\left(\sqrt[4]{-s},\sqrt{-\sqrt{-s}},\sqrt{2\sqrt{-s}}\right)=\Q\left(i,\sqrt{2},\sqrt[4]{-s}\right)=\Q\left(\zeta_8,\sqrt[4]{-s}\right).$$
Then, $G=\Gal(\Q(E[4])/\Q)$ can have $4$ isomorphism types according to the value of $s$:
\begin{enumerate}
\item If $s=\pm 1$ or $\pm 4$, then $\Q(E[4])=\Q(\zeta_8)$, and $G\cong (\Z/2\Z)^2$.
\item If $s=\pm t^2$, for some square-free $t\in \Z$ with $t\notin \{\pm 1,\pm 2\}$, then $\Q(E[4])=\Q(\zeta_8,\sqrt{t})$, and $G\cong (\Z/2\Z)^3$.
\item If $s=\pm 2t^2$, for some square-free $t\in \Z$, then $G\cong D_4$, which is non-abelian.
\item Otherwise, $G\cong D_4\times \Z/2\Z$, which is also non-abelian of order $16$.
\end{enumerate}
Thus, the only values of $s$ that make $\Q(E[4])/\Q$ abelian are $s=\pm t^2$, for some other square-free integer $t$, and this concludes the proof.
\end{proof}

The following proposition will treat the image of the $2$-adic Galois representation attached to an elliptic curve $E/\Q(j_{K,f})$ with CM by $\mathcal{O}_{K,f}$, with $j_{K,f}= 0$. But we first need a lemma.

\begin{lemma}\label{lem-abelian_conditions_j0}
     Let $N\geq 2$ and let  $G\subseteq \mathcal{N}_{\delta,1}(N)$ be a subgroup (so here $\phi=1$). Let
     \[
    \gamma' \in \left\{ c_1'=\left(\begin{array}{cc} 0 & 1\\ 1 & 0\\\end{array}\right),c_{-1}'=\left(\begin{array}{cc} 0 & -1\\ -1 & 0\\\end{array}\right) \right\} = \Gamma', 
    \]
     If $\gamma', c_{\delta,1}(a,b) \in G$, for some $a,b\in \Z/N\Z$, such that $\gamma' \cdot c_{\delta,1}(a,b) = c_{\delta,1}(a,b) \cdot \gamma'$, then $b\equiv 0 \bmod N$.  
     \end{lemma}
     \begin{proof} Let $\gamma'=c_\varepsilon'$ and suppose it commutes with $c_{\delta,1}(a,b)$. Then:
           \begin{align*}
       \left(\begin{array}{cc} a+b & b \\ \delta b & a\\\end{array}\right)\left(\begin{array}{cc} 0 & \varepsilon \\ \varepsilon & 0\\\end{array}\right) &\equiv \left(\begin{array}{cc} 0 & \varepsilon \\ \varepsilon & 0\\\end{array}\right)\left(\begin{array}{cc} a+b & b \\ \delta b & a\\\end{array}\right) \bmod N, \text{ and therefore}\\
        \left(\begin{array}{cc} \varepsilon b & \varepsilon(a+b) \\ \varepsilon a & \delta \varepsilon b\\\end{array}\right) &\equiv \left(\begin{array}{cc} \delta \varepsilon b & \varepsilon a\\ \varepsilon(a+b) & \varepsilon b\\\end{array}\right) \bmod N, 
    \end{align*}
    and therefore $b\equiv 0 \bmod N$, as claimed.
     \end{proof}

\begin{prop}\label{thm-jzero-intro}\label{cor-2adicimageforj0}\label{prop-j0andell2}
    Let $E/\Q$ be an elliptic curve with $j(E)=0$, and let $c\in \GQ$ be a complex conjugation, and $\gamma=\rho_{E,2^\infty}(c)$. Let $G_{E,2^\infty}$ be the image of $\rho_{E,2^\infty}$ and let $G_{E,K,2^\infty}=\rho_{E,2^\infty}(G_{\Q(\sqrt{-3})})$.
    \begin{itemize}
        \item[(i)] If $[\cC_{-1,1}(2^\infty):G_{E,K,2^\infty}]=3$, then the image $G_{E,2^n}$ is non-abelian if and only if $n\geq 2$. Here $G_{E,2}\cong \Z/2\Z$ and $E/\Q$ can be given by a Weierstrass model of the form $y^2=x^3+s^3$ for some non-zero $s\in\Z$.
        
        \item[(ii)] If $[\cC_{-1,1}(2^\infty):G_{E,K,2^\infty}]=1$, then the image $G_{E,2^n}$ is non-abelian if and only if $n\geq 1$. Here $G_{E,2}\cong S_3$ and $E/\Q$ can be given by a Weierstrass model of the form $y^2=x^3+s$ for some non-zero integer $s\in\Z$ that is not a cube of an integer.
    \end{itemize}
\end{prop}

\begin{proof}

 Let $E/\Q$ be an elliptic curve with $j(E)=0$. Thus, $E$ has CM by $\Of$, where $K=\Q(\sqrt{-3})$. Since $\Delta_K=-3$ and $f=1$, we set $\delta =-1$ and $\phi=1$. The image of $\rho_{E,2^\infty}$ is given in this case by \cite[Theorem 1.8]{lozano-galoiscm} and tells us that the image $G_{E,2^\infty} = \langle c_{\gamma'}, G_{E,K,2^\infty}\rangle$ where
    $$\gamma'\in \left\{ \left(\begin{array}{cc} 0 & 1\\ 1 & 0\\\end{array}\right),\left(\begin{array}{cc} 0 & -1\\ -1 & 0\\\end{array}\right)\right\},$$ such that $\gamma\equiv \gamma'\bmod 4$. 
    \begin{itemize}
        \item[(i)] If $[\cC_{-1,1}(2^\infty):G_{E,K,2^\infty}]=3$, then the image is
        \[
            G_{E,2^\infty} =\left\langle \gamma', -\operatorname{Id}, \left(\begin{array}{cc} 7 & 4\\ -4 & 3\\\end{array}\right), \left(\begin{array}{cc} 3 & 6\\ -6 & -3\\\end{array}\right)\right\rangle.
        \]
        First note that $G_{E,2}\equiv \{\gamma'\}\cong \Z/2\Z$ is abelian.  Now consider the image $G_{E,4}$. Since $c_{-1,1}(-3,6) \equiv c_{-1,1}(1,2) \bmod 4$ belongs to $G_{E,4}$ and $b\equiv 6\not\equiv 0 \bmod 4$, Lemma \ref{lem-abelian_conditions_j0} implies that $G_{E,4}$ must be non-abelian. Therefore, the image $G_{E,2^n}$ is non-abelian for $n\geq 2$.

        \item[(ii)] If $[\cC_{-1,1}(2^\infty):G_{E,K,2^\infty}]=1$, then the image is 
        \[
            G_{E,2^\infty} =\mathcal{N}_{-1,1}(2^\infty)=\left\langle \gamma', -\operatorname{Id}, \left(\begin{array}{cc} 7 & 4\\ -4 & 3\\\end{array}\right), \left(\begin{array}{cc} 2 & 1\\ -1 & 1\\\end{array}\right)\right\rangle.
        \]
        Observe that $G_{E,2^n}=\mathcal{N}_{-1,1}(2^n)$ for all $n\geq 1$, and $(\delta,\phi)\equiv (1,1)\bmod 2$. By Theorem \ref{thm-fullnormalizerabelian}, the image $G_{E,2^n}$ is non-abelian for all $n\geq 1$. 
    \end{itemize}
    The Weierstrass models for each case follow from Proposition 1.15.(iv) of \cite{zywina}. This concludes the proof.
\end{proof}

To end this section, we summarize our findings in one statement.

\begin{theorem}\label{thm-primepowercase}
    Let $E/\mathbb{Q}(j_{K,f})$ be an elliptic curve with CM by $\mathcal{O}_{K,f}$, $f\geq 1$, let $\ell$ be a prime, and let $n\geq 1$ such that the image $G_{E,\ell^n}$ of $\rho_{E,\ell^n}$ is abelian. Then, $\ell^n=2$, $3$, or $4$. 
\end{theorem}
\begin{proof}
    Let $E/\Q(j_{K,f})$ be an elliptic curve as in the statement, and suppose that $G_{E,\ell^n}$ is abelian. By Theorem \ref{thm-cmrep-intro-alvaro}, there is a $\Z/\ell^n\Z$-basis of $E[\ell^n]$ such that the image $G_{E,\ell^n}$ of $\rho_{E,\ell^n}$ is a subgroup of $\mathcal{N}_{\delta,\phi}(\Z/\ell^n\Z)$. By Theorem    \ref{thm-fullnormalizerabelian}, if $G_{E,\ell^n}\cong \mathcal{N}_{\delta,\phi}(\Z/\ell^n\Z)$, then $\ell^n=2$. So we may assume that $G_{E,\ell^n}$ is a proper subgroup of $ \mathcal{N}_{\delta,\phi}(\Z/\ell^n\Z)$.

    Further, if $\ell$ is not a divisor of $2\Delta_K f$ and $j_{K,f}\neq 0$, then Prop. \ref{prop-goodrednprimes} shows that $G_{E,p}$ is non-abelian. If $j_{K,f}=0$ and $\ell$ does not divide $2\Delta_K f=-6$, then Prop. \ref{prop-jzero-goodredn} shows that $G_{E,\ell}$ is non-abelian. Thus, we may assume that $\ell$ is a divisor of $2\Delta_K f$. 

    If $\ell>2$ is a divisor of $\Delta_K f$, and $G_{E,\ell^n}$ is abelian, then Prop. \ref{prop-abelianforp3} shows that $j_{K,f}=0$ and $\ell^n=3$, and $G_{E,9}$ is never abelian.

    It remains to treat the case of $\ell=2$. If $j_{K,f}\neq 0,1728$ and $G_{E,2^n}$ is abelian, then Prop. \ref{prop-jnot01728andell2} shows that $2^n=2$ and $G_{E,4}$ is never abelian. If $j_{K,f}=1728$, then Prop. \ref{prop-j1728andell2} shows that $2^n\leq 4$. And if $j_{K,f}=0$, then Prop. \ref{prop-j0andell2} shows that $2^n=2$.

    Hence, in all cases $\ell^n=2$, $3$, or $4$, and this concludes the proof.
\end{proof}


\section{Proof that the image $G_{E,N}$  is non-abelian for $N\geq 5$}\label{proof-thm-main1}\label{sec-proof-main-theorem1}

In Section \ref{applying_results}, we have shown that if the image $G_{E,N}$ is abelian for $N=\ell^n$, for some prime $\ell$, then $N=2,3$, or $4$. Here we shall show that $G_{E,N}$ is non-abelian for all $N\geq 5$. We shall begin by showing that $G_{E,N}$ is non-abelian for $N=6$.

\begin{prop}\label{n=6}
    Let $E/\mathbb{Q}(j_{K,f})$ be an elliptic curve with CM by $\mathcal{O}_{K,f}$, $f\geq 1$. Then, the image $G_{E,6}$ is non-abelian. 
\end{prop}

\begin{proof}
    Recall that $\mathcal{N}_{\delta,\phi}(6)$ is never abelian by Theorem \ref{abelian_n=2}. Suppose there exists an elliptic curve $E/\mathbb{Q}(j_{K,f})$ with CM by $\mathcal{O}_{K,f}$, $f\geq 1$, such that the image $G_{E,6}\subsetneq \mathcal{N}_{\delta,\phi}(6)$ is abelian. If we reduce the image $G_{E,6} \bmod 3$, then we get the mod-$3$ image $G_{E,3}$, which therefore must also be abelian, and similarly $G_{E,2} \equiv G_{E,6} \bmod 2$ must be abelian as well. If the image $G_{E,N}$ is abelian for $N=3$, then Theorem \ref{thm-primepowercase} (more concretely Prop. \ref{cor-abelianforp3}) tells us that $j_{K,f}=0$, so let us assume that $j_{K,f}=0$. Moreover, Prop. \ref{cor-abelianforp3} shows that if $j_{K,f}=0$ and $G_{E,3}$ is abelian, then we must have $E/\Q$ given by $y^2=x^3+d$ with $d$ a non-zero integer such that $-4d$ is a perfect cube.

    Since $G_{E,2}$ must be abelian, Prop. \ref{cor-2adicimageforj0} shows that the mod-$2$ image of $E$ must have index $3$ in $\mathcal{N}_{\delta,\phi}(2)$, and $E$ is given by $y^2=x^3+d$ for some non-zero integer $d$ which is a perfect cube. However, if $d$ is a cube, and $-4d$ is also a cube, it would follow that $-4$ is a cube in $\Z$, so we have reached a contradiction.  Hence $G_{E,6}$ cannot be abelian.
\end{proof}

Now we will show that $G_{E,N}$ is non-abelian for all $N\geq 5$ and complete the proof of  Theorem \ref{thm-main}.

\begin{proof}[Proof of Theorem \ref{thm-main}]
    Let $N\geq 5$, let $G_{E,N}$ be the image of $\rho_{E,N}$, and suppose for a contradiction that $G_{E,N}$ is abelian. Let $\ell$ be a prime divisor of $N$, and let $n=\nu_\ell(N)$ be the $\ell$-adic valuation of $N$, so that $\ell^n$ is the highest power of $\ell$ that divides $N$. In particular, $G_{E,\ell^n}$ is also abelian. Theorem \ref{thm-primepowercase} shows that if $G_{E,\ell^n}$ is abelian, then $\ell^n=2,3,$ or $4$. Thus,
    \begin{itemize}
        \item[(a)] $\ell = 2$ or $3$,
        \item[(b)] if $\ell=3$, then $n=1$, and
        \item[(c)] if $\ell=2$ then, $1\leq n \leq 2$,
    \end{itemize}
    Since $N\geq 5$, it follows that $N=6$ or $12$. However, if $G_{E,12}$ is abelian, then so is $G_{E,6}$ because $6$ is a divisor of $12$ and there is a natural projection map $\mathcal{N}_{\delta,\phi}(\Z/12\Z)\to \mathcal{N}_{\delta,\phi}(\Z/6\Z)$. But Prop. \ref{n=6} shows that $G_{E,N}$ is non-abelian when $N=6$. Therefore, $G_{E,N}$ is non-abelian for any $N$ divisible by 6, and $G_{E,N}$ cannot be abelian for any $N\geq 5$. 

    About the specific conditions for abelian image when $N=2$, $3$, or $4$: 
    \begin{enumerate}
        \item If $N=2$ and $j_{K,f}\neq 0,1728$, then Prop. \ref{prop-jnot01728andell2} shows that $G_{E,2}$ is abelian if and only if $G_{E,2}\cong \Z/2\Z$ and either
      \begin{enumerate}
            \item $\Delta_Kf^2\equiv 0 \bmod 4$, or 
            \item  $\Delta_K\equiv 1 \bmod 8$ and $f\equiv 1 \bmod 2$.
        \end{enumerate}
        If $j_{K,f}=1728$, then $G_{E,2}$ is always abelian (see Prop. \ref{prop-j1728andell2}) with $G_{E,2}\cong \Z/2\Z$ or $\{0\}$ according to whether $E$ is given by $y^2=x^3-dx$ with $d$ a non-square or a square in $\Z$. Prop. \ref{prop-j0andell2} shows that if $j_{K,f}=0$, then $G_{E,2}$ is abelian if and only if $E/\Q$ is given by a form $y^2=x^3+d$ with $d$ a cube in $\Z$.

        \item If $N=3$, then Prop. \ref{prop-abelianforp3} shows that the only possibility is that $j_{K,f}=0$ and $E/\Q$ is given by $y^2=x^3+d$ such that $-4d$ is a cube in $\Z$.

        \item Finally, if $N=4$, then Prop. \ref{prop-j1728andell2} shows that $j_{K,f}=1728$, the index $[\mathcal{N}_{-1,0}(\Z_{2^\infty}):G_{E,2^\infty}]$ is $2$ or $4$, and Lemma \ref{lem-4abelianj1728} describes the Weierstrass equations of elliptic curves with $j_{K,f}=1728$ and abelian $4$-division field.
    \end{enumerate}
This concludes the proof of the main Theorem \ref{thm-main}.
\end{proof}



\section{Cyclotomic division fields}\label{proof-thm-main2}\label{sec-proof-main-thm2}

In this section we prove Theorem \ref{thm-main2}, i.e., we classify elliptic curves with CM that are defined over their minimal field of definition and the $N$-division field coincides with the $N$-th cyclotomic extension of $\Q(j_{K,f})$.

\begin{theorem}
    Let $E/\mathbb{Q}(j_{K,f})$ be an elliptic curve with CM by $\mathcal{O}_{K,f}$, $f \geq 1$, and let $N\geq 2$. Then, $\mathbb{Q}(j_{K,f}, E[N]) = \mathbb{Q}(j_{K,f},\zeta_N)$ if and only if 
    \begin{enumerate}
        \item $N=2$, and $j_{K,f}=1728$ with $E/\Q$ given by a form $y^2=x^3-d^2x$ for some integer $d$, or
        \item $N=3$, and $j_{K,f}=0$ with $E/\Q$ given by a form $y^2=x^3+d$ for some integer $d$ such that $d$ or $-3d$ is a square and $-4d$ is a cube. 
    \end{enumerate}
\end{theorem} 

\begin{proof}
    Let $E/\Q(j_{K,f})$ with CM, and suppose $N\geq 2$ such that $\Q(j_{K,f},E[N]) = \Q(j_{K,f},\zeta_N)$. In particular, $\Q(j_{K,f},E[N])/\Q(j_{K,f})$ is abelian and so is $G_{E,N}$, and therefore, by Theorem \ref{thm-main}, we must have $N=2$, $3$, or $4$.
    \begin{enumerate}
        \item If $N=2$, then $\Q(j_{K,f},\zeta_N)=\Q(j_{K,f})$ and therefore $G_{E,2}$ is trivial. By Theorem \ref{thm-main}, this can only occur when $j_{K,f}=1728$, and $E/\Q$ is given by a model $y^2=x^3-d^2x$ for some non-zero integer $d$.

        \item If $N=3$, then $\Q(j_{K,f},\zeta_N)=\Q(j_{K,f},\sqrt{-3})$ is a trivial or quadratic extension of $\Q(j_{K,f})$.  Prop. \ref{prop-abelianforp3} shows that the only possibility is that $j(E)=0$ and $E/\Q$ is given by $y^2=x^3+d$ such that $-4d$ is a cube in $\Z$. Moreover, $G_{E,3}\cong \Z/2\Z$ occurs precisely whenever $d$ or $-3d$ is a square and $-4d$ is a cube. Finally, since $K=\Q(\sqrt{-3})\subseteq \Q(E[3])$ (by \cite[Lemma 2.2]{lozano-galoiscm}), it follows that $\Q(E[3])=K=\Q(\sqrt{-3})=\Q(\zeta_3)$ in such case.

        \item If $N=4$, then $\Q(j_{K,f},\zeta_N)=\Q(j_{K,f},i)$ is a trivial or quadratic extension of $\Q(j_{K,f})$. However, Theorem \ref{thm-main} and Prop. \ref{prop-j1728andell2} show that if $G_{E,4}$ is abelian, then $G_{E,4}\cong (\Z/2\Z)^2$ or $(\Z/2\Z)^3$, and it is never trivial or isomorphic to $\Z/2\Z.$ Hence, $\Q(j_{K,f},E[4])=\Q(j_{K,f},\zeta_4)$ is impossible.
    \end{enumerate}
    This concludes the proof of the theorem.
\end{proof}


\bibliography{bibliography}
\bibliographystyle{plain}


\end{document}